\documentclass{birkjour}
 \newtheorem{thm}{Theorem}[section]
 
 \newtheorem{lem}[thm]{Lemma}
 \newtheorem{prop}[thm]{Proposition}
 \theoremstyle{definition}
 
 \theoremstyle{remark}
 \newtheorem{rem}[thm]{Remark}
 
 \numberwithin{equation}{section}

\newcommand{\mymod}[3]{#1 \equiv #2 \kern -0.5em \pmod{#3}}
\newcommand{\mynotmod}[3]{#1 \not \equiv #2 \kern -0.6em \pmod{#3}}
\renewcommand{\theequation}{\thesection.\arabic{equation}}

\begin{document}

%
%
%
%
%
%
%
%
%

\title[On the third-order Horadam and geometric mean sequences]{On the third-order Horadam and geometric mean sequences}

\author[G. Cerda-Morales]{Gamaliel Cerda-Morales}
\address{Instituto de Matem\'aticas, Pontificia Universidad Cat\'olica de Valpara\'iso, Blanco Viel 596, Valpara\'iso, Chile.}
\email{gamaliel.cerda.m@mail.pucv.cl}




\begin{abstract}
This paper, in considering aspects of the geometric mean sequence, offers new results connecting generalized Tribonacci and third-order Horadam numbers which are established and then proved independently.
\end{abstract}

\maketitle
\section {Introduction}


The Horadam numbers have many interesting properties and applications in many fields of science (see, e.g., \cite{Lar1,Lar2}). The Horadam numbers $H_{n}(a,b;r,s)$ or $H_{n}$ are defined by the recurrence relation
\begin{equation}\label{e1}
H_{0}=a,\ H_{1}=b,\ H_{n+2}=rH_{n+1}+sH_{n},\ n\geq0.
\end{equation}
Another important sequence is the generalized Fibonacci sequence $\{h_{n}^{(3)}\}_{n\in \mathbb{N}}$.  This sequence is defined by the recurrence relation $h_{n+2}=rh_{n+1}+sh_{n}$, with $h_{0}=0$, $h_{1}=1$ and $n\geq0$.

In \cite{Sha,Wa} the Horadam recurrence relation (\ref{e1}) is extended to higher order recurrence relations and the basic list of identities provided by A. F. Horadam is expanded and extended to several identities for some of the higher order cases. In fact, third-order Horadam numbers, $\{H_{n}^{(3)}(a,b,c;r,s,t)\}$, and generalized Tribonacci numbers, $\{h_{n}^{(3)}(0,1,r;r,s,t)\}_{n\geq0}$, are defined by
\begin{equation}\label{ec:5}
H_{n+3}^{(3)}=rH_{n+2}^{(3)}+sH_{n+1}^{(3)}+tH_{n}^{(3)},\ H_{0}^{(3)}=a,\ H_{1}^{(3)}=b,\ H_{2}^{(3)}=c,\ n\geq0,
\end{equation}
and 
\begin{equation}\label{ec:6}
h_{n+3}^{(3)}=rh_{n+2}^{(3)}+sh_{n+1}^{(3)}+th_{n}^{(3)},\ h_{0}^{(3)}=0,\ h_{1}^{(3)}=1,\ h_{2}^{(3)}=r,\ n\geq0,
\end{equation}
respectively.

Some of the following properties given for third-order Horadam numbers and generalized Tribonacci numbers are revisited in this paper (for more details, see \cite{Cer1,Sha,Wa}). 
\begin{equation}\label{e4}
H_{n+m}^{(3)}=h_{n}^{(3)}H_{m+1}^{(3)}+\left(sh_{n-1}^{(3)}+th_{n-2}^{(3)}\right)H_{m}^{(3)}+th_{n-1}^{(3)}H_{m-1}^{(3)},
\end{equation}
\begin{equation}\label{e5}
\left(h_{n}^{(3)}\right)^{2}+s\left(h_{n-1}^{(3)}\right)^{2}+2th_{n-1}^{(3)}h_{n-2}^{(3)}=h_{2n-1}^{(3)}
\end{equation}
and
\begin{equation}\label{ec5}
\left(H_{n}^{(3)}\right)^{2}+s\left(H_{n-1}^{(3)}\right)^{2}+2tH_{n-1}^{(3)}H_{n-2}^{(3)}=\left\lbrace 
\begin{array}{c}
cH_{2n-2}^{(3)}+\left(sb+ta\right)H_{2n-3}^{(3)}\\
+tbH_{2n-4}^{(3)},
\end{array}
\right\rbrace,
\end{equation}
where $n\geq 2$ and $m\geq 1$.

As the elements of this Tribonacci-type number sequence provide third order iterative relation, its characteristic equation is $x^{3}-rx^{2}-sx-t=0$, whose roots are $\alpha=\frac{r}{3}+A+B$, $\omega_{1}=\frac{r}{3}+\epsilon A+\epsilon^{2} B$ and $\omega_{2}=\frac{r}{3}+\epsilon^{2}A+\epsilon B$, where $$A=\sqrt[3]{\frac{r^{3}}{27}+\frac{rs}{6}+\frac{t}{2}+\sqrt{\Delta}},\ B=\sqrt[3]{\frac{r^{3}}{27}+\frac{rs}{6}+\frac{t}{2}-\sqrt{\Delta}},$$ with $\Delta=\Delta(r,s,t)=\frac{r^{3}t}{27}-\frac{r^{2}s^{2}}{108}+\frac{rst}{6}-\frac{s^{3}}{27}+\frac{t^{2}}{4}$ and $\epsilon=-\frac{1}{2}+\frac{i\sqrt{3}}{2}$. 

In this paper, $\Delta>0$, then the cubic equation $x^{3}-rx^{2}-sx-t=0$ has one real and two nonreal solutions, the latter being conjugate complex. Thus, the Binet formula for the third-order Horadam numbers can be expressed as:
\begin{equation}\label{eq:8}
H_{n}^{(3)}=\frac{P\alpha^{n}}{(\alpha-\omega_{1})(\alpha-\omega_{2})}-\frac{Q\omega_{1}^{n}}{(\alpha-\omega_{1})(\omega_{1}-\omega_{2})}+\frac{R\omega_{2}^{n}}{(\alpha-\omega_{2})(\omega_{1}-\omega_{2})},
\end{equation}
where the coefficients are $P=c-(\omega_{1}+\omega_{2})b+\omega_{1}\omega_{2}a$, $Q=c-(\alpha+\omega_{2})b+\alpha\omega_{2}a$ and $R=c-(\alpha+\omega_{1})b+\alpha\omega_{1}a$.

In particular, if $a=0$, $b=1$ and $c=r$, we obtain $H_{n}^{(3)}=h_{n}^{(3)}$. In this case, $P=\alpha$, $Q=\omega_{1}$ and $R=\omega_{2}$ in Eq. (\ref{eq:8}). In fact, the third-order Horadam sequence is the generalization of the well-known sequences like Tribonacci, Padovan, Narayana and third-order Jacobsthal (see \cite{Cer}). 

Consider the (scaled) geometric mean sequence $\{g_{n}^{(3)}(a,b,c;\epsilon)\}_{n=0}^{\infty}$ defined, given $g_{0}^{(3)}=a$, $g_{1}^{(3)}=b$ and $g_{2}^{(3)}=c$, through the recurrence
\begin{equation}\label{e0}
g_{n+3}^{(3)}=\epsilon \sqrt[3]{g_{n+2}^{(3)}g_{n+1}^{(3)}g_{n}^{(3)}},\ n\geq0,
\end{equation}
where $\epsilon \in \mathbb{N}$ is a scaling constant. In this paper we assume $\epsilon=1$ and begin by finding the growth rate of the sequence
\begin{equation}\label{e00}
\left\lbrace g_{n}^{(3)}\right\rbrace=\left\lbrace a, b, c, \left(abc\right)^{\frac{1}{3}}, \left(ab^{4}c^{4}\right)^{\frac{1}{9}}, \left(a^{4}b^{7}c^{16}\right)^{\frac{1}{27}}, \left(a^{16}b^{28}c^{37}\right)^{\frac{1}{81}}, \cdots \right\rbrace 
\end{equation}
using two alternative approaches, one of which are routine with the other based on a connection between this sequence and the generalized Tribonacci sequence $\{T_{n}^{(3)}\}=\{H_{n}^{(3)}(0,0,1;1,3,9)\}$ defined by
\begin{equation}\label{ecua:6}
T_{n+3}^{(3)}=T_{n+2}^{(3)}+3T_{n+1}^{(3)}+9T_{n}^{(3)},\ T_{0}^{(3)}=T_{1}^{(3)}=0,\ T_{2}^{(3)}=1,
\end{equation}
discernible in the powers of $a$, $b$ and $c$ in  Eq. (\ref{e00}). Other results follow accordingly, with generalized Tribonacci numbers expressed in terms of families of parameterized third-order Horadam numbers in two particular identities that are established in different ways.


\section {Growth Rate of $\{g_{n}^{(3)}(a,b,c;1)\}_{n=0}^{\infty}$}

We begin by showing that the growth rate of the sequence $\{g_{n}^{(3)}(a,b,c;1)\}_{n=0}^{\infty}$ is 1. That is to say the following.

\begin{thm}\label{t1}
The sequence  $\{g_{n}^{(3)}(a,b,c;1)\}_{n=0}^{\infty}$ grows according to
\begin{equation}\label{p1}
\lim_{n\rightarrow \infty}\frac{g_{n+1}^{(3)}(a,b,c;1)}{g_{n}^{(3)}(a,b,c;1)}=1.
\end{equation}
\end{thm}
As alluded to in the Introduction, we approach the proof of Theorem \ref{t1} in two ways.

\textbf{Method I}. This is elementary, and parallels that seen in \cite{Shi} for the case $\{g_{n}^{(2)}(a,b;1)\}_{n=0}^{\infty}$.
\begin{proof}
Writing $\mu=\lim_{n\rightarrow \infty}\frac{g_{n+1}^{(3)}}{g_{n}^{(3)}}\in \mathbb{R}^{+}$ and using Eq. (\ref{e0}), then
\begin{align*}
\mu^{3}&=\left(\lim_{n\rightarrow \infty}\frac{g_{n+3}^{(3)}}{g_{n+2}^{(3)}}\right)^{3}\\
&=\lim_{n\rightarrow \infty}\frac{g_{n+2}^{(3)}g_{n+1}^{(3)}g_{n}^{(3)}}{\left(g_{n+2}^{(3)}\right)^{3}}\\
&=\lim_{n\rightarrow \infty}\frac{g_{n+1}^{(3)}g_{n}^{(3)}g_{n+1}^{(3)}}{g_{n+2}^{(3)}g_{n+2}^{(3)}g_{n+1}^{(3)}}=\frac{1}{\mu^{3}}.
\end{align*}
So that $\mu^{6}=1$ of which $\mu=1$ is a real positive solution.
\end{proof}

\textbf{Method II}.  This is rather more interesting, since it relies on a closed form for the sequence term $\{g_{n}^{(3)}(a,b,c;1)\}_{n=0}^{\infty}$ in which generalized Tribonacci numbers make an appearance. A previous result is demonstrated

\begin{lem}\label{lem}
For $n\geq 0$, 
\begin{equation}\label{gam}
g_{n+2}^{(3)}(a,b,c;1)=\left(a^{T_{n+1}^{(3)}}b^{T_{n+1}^{(3)}+3T_{n}^{(3)}}c^{T_{n+2}^{(3)}}\right)^{3^{-n}},
\end{equation}
where $T_{n}^{(3)}$ is as in Eq. (\ref{ecua:6}).
\end{lem}
\begin{proof}
To prove Eq. (\ref{gam}), let us use the induction on $n$. If $n=0$, the proof is obvious since that $T_{0}^{(3)}=0=T_{1}^{(3)}=0$, $T_{2}^{(3)}=1$ and $$g_{2}^{(3)}(a,b,c;1)=c=\left(a^{T_{1}^{(3)}}b^{T_{1}^{(3)}+3T_{0}^{(3)}}c^{T_{2}^{(3)}}\right)^{3^{-0}}.$$ Let us assume that Eq. (\ref{gam}) holds for all values $m$ less than or equal $n$. Now we have to show that the result is true for $n+1$:
\begin{align*}
\left(g_{(n+1)+2}^{(3)}\right)^{3}&=g_{n+2}^{(3)}g_{n+1}^{(3)}g_{n}^{(3)}\\
&=\left(a^{T_{n+1}^{(3)}}b^{T_{n+1}^{(3)}+3T_{n}^{(3)}}c^{T_{n+2}^{(3)}}\right)^{3^{-n}}\times \left(a^{T_{n}^{(3)}}b^{T_{n}^{(3)}+3T_{n-1}^{(3)}}c^{T_{n+1}^{(3)}}\right)^{3^{-(n-1)}}\\
&\ \ \times \left(a^{T_{n-1}^{(3)}}b^{T_{n-1}^{(3)}+3T_{n-2}^{(3)}}c^{T_{n}^{(3)}}\right)^{3^{-(n-2)}}\\
&=\left\lbrace\begin{array}{c}a^{T_{n+1}^{(3)}+3T_{n}^{(3)}+9T_{n-1}^{(3)}}\\
\times b^{\left(T_{n+1}^{(3)}+3T_{n}^{(3)}+9T_{n-1}^{(3)}\right)+3\left(T_{n}^{(3)}+3T_{n-1}^{(3)}+9T_{n-2}^{(3)}\right)}\\
\times c^{T_{n+2}^{(3)}+3T_{n+1}^{(3)}+9T_{n}^{(3)}}\end{array}\right\rbrace^{3^{-n}}\\
&=\left(a^{T_{n+2}^{(3)}}b^{T_{n+2}^{(3)}+3T_{n+1}^{(3)}}c^{T_{n+3}^{(3)}}\right)^{3^{-n}}.
\end{align*}
Then, using Eq. (\ref{ecua:6}) we obtain the result.
\end{proof}

The proof of Theorem \ref{t1} is now immediate.
\begin{proof}
From Eq. (\ref{gam}), we have 
\begin{align*}
\lim_{n\rightarrow \infty}\frac{g_{n+3}^{(3)}}{g_{n+2}^{(3)}}&=\lim_{n\rightarrow \infty}\frac{\left(a^{T_{n+2}^{(3)}}b^{T_{n+2}^{(3)}+3T_{n+1}^{(3)}}c^{T_{n+3}^{(3)}}\right)^{3^{-(n+1)}}}{\left(a^{T_{n+1}^{(3)}}b^{T_{n+1}^{(3)}+3T_{n}^{(3)}}c^{T_{n+2}^{(3)}}\right)^{3^{-n}}}\\
&=\lim_{n\rightarrow \infty}\left\lbrace \begin{array}{c}a^{T_{n+2}^{(3)}-3T_{n+1}^{(3)}}b^{\left(T_{n+2}^{(3)}-3T_{n+1}^{(3)}\right)+3\left(T_{n+1}^{(3)}-3T_{n}^{(3)}\right)}\\
\times c^{T_{n+3}^{(3)}-3T_{n+2}^{(3)}} \end{array}\right\rbrace^{3^{-(n+1)}}\\
&=\lim_{n\rightarrow \infty}a^{\frac{T_{n+2}^{(3)}-3T_{n+1}^{(3)}}{3^{n+1}}}b^{\frac{T_{n+2}^{(3)}-9T_{n}^{(3)}}{3^{n+1}}}c^{\frac{T_{n+3}^{(3)}-3T_{n+2}^{(3)}}{3^{n+1}}}.
\end{align*}
Now, using Eq. (\ref{eq:8}) and $\{T_{n}^{(3)}\}=\{H_{n}^{(3)}(0,0,1;1,3,9)\}$, we obtain the Binet formula 
\begin{align*}
T_{n+1}^{(3)}&=\frac{1}{6}\left[3^{n}+\left(\frac{-1-i\sqrt{2}}{2}\right)(-1+i\sqrt{2})^{n}+\left(\frac{-1+i\sqrt{2}}{2}\right)(-1-i\sqrt{2})^{n}\right]\\
&=\frac{1}{6}\left[3^{n}+\frac{3}{2}\left((-1+i\sqrt{2})^{n-1}+(-1-i\sqrt{2})^{n-1}\right)\right]\\
&=\frac{1}{6}\left[3^{n}+V_{n}^{(2)}\right],\ n\geq 1,
\end{align*}
where $V_{n}^{(2)}=\frac{3}{2}\left((-1+i\sqrt{2})^{n-1}+(-1-i\sqrt{2})^{n-1}\right)$. In fact, $V_{n}^{(2)}$ satisfies the following properties $V_{n+2}^{(2)}=-2V_{n+1}^{(2)}-3V_{n}^{(2)}$, $V_{0}^{(2)}=-1$, $V_{1}^{(2)}=3$ and $$\lim_{n\rightarrow \infty}\frac{V_{n}^{(2)}}{3^{n}}=\lim_{n\rightarrow \infty}\frac{1}{2}\left(\left(\frac{-1+i\sqrt{2}}{3}\right)^{n-1}+\left(\frac{-1-i\sqrt{2}}{3}\right)^{n-1}\right)=0.$$ Then, we have
\begin{equation}\label{teo1}
\lim_{n\rightarrow \infty}\frac{g_{n+3}^{(3)}}{g_{n+2}^{(3)}}=\lim_{n\rightarrow \infty}\left\lbrace \begin{array}{c}a^{\frac{1}{6}\left(\frac{V_{n+1}^{(2)}-3V_{n}^{(2)}}{3^{n+1}}\right)} \times b^{\frac{1}{6}\left(\frac{V_{n+1}^{(2)}-9V_{n-1}^{(2)}}{3^{n+1}}\right)}\\
\times c^{\frac{1}{2}\left(\frac{V_{n+2}^{(2)}-3V_{n+1}^{(2)}}{3^{n+2}}\right)}\end{array}\right\rbrace=1
\end{equation}
as required in Eq. (\ref{p1}).
\end{proof}

\section{Generalized Tribonacci and Third-order Horadam Identities}

We now develop a couple of identities that link terms of the generalized Tribonacci sequence $T_{n}^{(3)}$ with those of particular third-order Horadam sequences; these are then generalized, with proofs given.

The power product recurrence
\begin{equation}\label{n1}
z_{n+3}^{(3)}=\left(z_{n+2}^{(3)}\right)^{r}\left(z_{n+1}^{(3)}\right)^{s}\left(z_{n}^{(3)}\right)^{t},\ n\geq 0,
\end{equation}
with initial values $z_{0}^{(3)}=a$, $z_{1}^{(3)}=b$ and $z_{2}^{(3)}=c$, is known to produce a sequence $\{z_{n}^{(3)}\}_{n\geq 0}=\{z_{n}^{(3)}(a,b,c;r,s,t)\}$ for which
\begin{equation}\label{n2}
z_{n}^{(3)}(a,b,c;r,s,t)=a^{H_{n}^{(3)}(1,0,0;r,s,t)}b^{H_{n}^{(3)}(0,1,0;r,s,t)}c^{H_{n}^{(3)}(0,0,1;r,s,t)},\end{equation}
with $H_{n}^{(3)}$ as in Eq. (\ref{eq:8}).

It was first put forward by Bunder in 1975 \cite{Bu} for the case $t=0$, having recently been proved
inductively and generalized by Larcombe and Bagdasar in \cite{Lar3}. Since our recursion in Eq. (\ref{e0}) (with $\epsilon=1$) is the case $r=s=t=\frac{1}{3}$ of Eq. (\ref{n1}) we can infer immediately from Eq. (\ref{n2}) and Lemma \ref{lem} that
\begin{equation}\label{n3}
\begin{aligned}
g_{n}^{(3)}(a,b,c;1)&=\left(a^{T_{n-1}^{(3)}}b^{T_{n-1}^{(3)}+3T_{n-2}^{(3)}}c^{T_{n}^{(3)}}\right)^{3^{-(n-2)}}\\
&=z_{n}^{(3)}\left(a,b,c;\frac{1}{3},\frac{1}{3},\frac{1}{3}\right)\\
&=a^{H_{n}^{(3)}\left(1,0,0;\frac{1}{3},\frac{1}{3},\frac{1}{3}\right)}b^{H_{n}^{(3)}\left(0,1,0;\frac{1}{3},\frac{1}{3},\frac{1}{3}\right)}c^{H_{n}^{(3)}\left(0,0,1;\frac{1}{3},\frac{1}{3},\frac{1}{3}\right)},
\end{aligned}
\end{equation}
delivering
\begin{equation}\label{i1}
T_{n-1}^{(3)}=3^{n-2}H_{n}^{(3)}\left(1,0,0;\frac{1}{3},\frac{1}{3},\frac{1}{3}\right),\ n\geq 1,
\end{equation}
\begin{equation}\label{i2}
T_{n-1}^{(3)}+3T_{n-2}^{(3)}=3^{n-2}H_{n}^{(3)}\left(0,1,0;\frac{1}{3},\frac{1}{3},\frac{1}{3}\right),\ n\geq 2
\end{equation}
and
\begin{equation}\label{i3}
T_{n}^{(3)}=3^{n-2}H_{n}^{(3)}\left(0,0,1;\frac{1}{3},\frac{1}{3},\frac{1}{3}\right),\ n\geq 0,
\end{equation}
which we believe are new relations in that they express generalized Tribonacci numbers $T_{n}^{(3)}$ in terms of third-order Horadam numbers $H_{n}^{(3)}\left(a,b,c;r,s,t\right)$ of Eq. (\ref{eq:8}).

\begin{rem}
We remark, for completeness, that values $r=s=t=1$ in Eq. (\ref{n1}) yield, by Eq. (\ref{n2}), 
\begin{align*}
z_{n}^{(3)}(a,b,c;1,1,1)&=a^{H_{n}^{(3)}(1,0,0;1,1,1)}b^{H_{n}^{(3)}(0,1,0;1,1,1)}c^{H_{n}^{(3)}(0,0,1;1,1,1)}\\
&=a^{T_{n-2}}b^{T_{n-2}+T_{n-3}}c^{T_{n-1}},\ n\geq 3,
\end{align*}
where $\{T_{n}\}_{n\geq 0}=\{0,1,1,2,4,7,... \}$ is the classic Tribonacci sequence defined by $$T_{n}=T_{n-1}+T_{n-2}+T_{n-3},\ T_{0}=0,\ T_{1}=T_{2}=1, n\geq 3.$$
\end{rem}

Relations (\ref{i1}), (\ref{i2}) and (\ref{i3}) emerge naturally as a result of what we know about sequence $\{z_{n}^{(3)}(a,b,c;r,s,t)\}_{n\geq 0}$. It is readily seen, however, that they are merely instances of general ones. In fact, setting $a=1$ and $b=c=0$, Eq. (\ref{ec:5}) gives
\begin{equation}\label{n5}
\begin{aligned}
H_{n}^{(3)}&=\frac{P\alpha^{n}}{(\alpha-\omega_{1})(\alpha-\omega_{2})}-\frac{Q\omega_{1}^{n}}{(\alpha-\omega_{1})(\omega_{1}-\omega_{2})}+\frac{R\omega_{2}^{n}}{(\alpha-\omega_{2})(\omega_{1}-\omega_{2})}\\
&=\frac{t\alpha^{n-1}}{(\alpha-\omega_{1})(\alpha-\omega_{2})}-\frac{t\omega_{1}^{n-1}}{(\alpha-\omega_{1})(\omega_{1}-\omega_{2})}+\frac{t\omega_{2}^{n-1}}{(\alpha-\omega_{2})(\omega_{1}-\omega_{2})},
\end{aligned}
\end{equation}
where the coefficients are $P=\omega_{1}\omega_{2}$, $Q=\alpha\omega_{2}$ and $R=\alpha\omega_{1}$, and choosing further, for arbitrary $\lambda$, $r(\lambda)=\lambda$, $s(\lambda)=3\lambda^{2}$ and $t(\lambda)=9\lambda^{3}$ ($\alpha(\lambda)=3\lambda$, $\omega_{1}(\lambda)=-\lambda+i\lambda\sqrt{2}$ and $\omega_{2}(\lambda)=-\lambda-i\lambda\sqrt{2}$), then
\begin{align*}
H_{n}^{(3)}(1,0,0;\lambda,3\lambda^{2},9\lambda^{3})&=9\lambda^{n}\left\lbrace \begin{array}{c} \frac{\alpha^{n-1}}{(\alpha-\omega_{1})(\alpha-\omega_{2})}-\frac{\omega_{1}^{n-1}}{(\alpha-\omega_{1})(\omega_{1}-\omega_{2})}\\
+\frac{\omega_{2}^{n-1}}{(\alpha-\omega_{2})(\omega_{1}-\omega_{2})}\end{array}\right\rbrace\\
&=9\lambda^{n}T_{n-1}^{(3)}
\end{align*}
by Eq. (\ref{eq:8}) and we have a generalized form of Eq. (\ref{i1}), parameterized by $\lambda$, which recovers it when $\lambda=\frac{1}{3}$:

\begin{prop}\label{prop1}
For $n\geq 2$, we have
\begin{equation}\label{iden1}
T_{n-1}^{(3)}=\frac{1}{9\lambda^{n}}H_{n}^{(3)}(1,0,0;\lambda,3\lambda^{2},9\lambda^{3}).
\end{equation}
\end{prop}

In a similar fashion, Eq. (\ref{ec:5}) gives 
\begin{equation}\label{n6}
H_{n}^{(3)}(0,0,1;r,s,t)=\left\lbrace \begin{array}{c} \frac{\alpha^{n}}{(\alpha-\omega_{1})(\alpha-\omega_{2})}-\frac{\omega_{1}^{n}}{(\alpha-\omega_{1})(\omega_{1}-\omega_{2})}\\
+\frac{\omega_{2}^{n}}{(\alpha-\omega_{2})(\omega_{1}-\omega_{2})}\end{array}\right\rbrace
\end{equation}
(the general term of the fundamental sequence mentioned earlier), with 
\begin{align*}
H_{n}^{(3)}(0,0,1;\lambda,3\lambda^{2},9\lambda^{3})&=\lambda^{n-2}\left\lbrace \begin{array}{c} \frac{\alpha^{n}}{(\alpha-\omega_{1})(\alpha-\omega_{2})}-\frac{\omega_{1}^{n}}{(\alpha-\omega_{1})(\omega_{1}-\omega_{2})}\\
+\frac{\omega_{2}^{n}}{(\alpha-\omega_{2})(\omega_{1}-\omega_{2})}\end{array}\right\rbrace\\
&=\lambda^{n-2}T_{n}^{(3)}
\end{align*}
and so a general form of Eq. (\ref{i3}) which latter is also reproduced for the same value $\lambda=\frac{1}{3}$:

\begin{prop}\label{prop2}
For $n\geq 1$, we have
\begin{equation}\label{iden2}
T_{n}^{(3)}=\frac{1}{\lambda^{n-2}}H_{n}^{(3)}(0,0,1;\lambda,3\lambda^{2},9\lambda^{3}).
\end{equation}
\end{prop}

Here, we will just prove Eq. (\ref{iden2}) since (\ref{iden1}) can be dealt with in the same manner. We use two different ways to proof the result.

\begin{proof}[Proof 1 of Eq. (\ref{iden2})]
This proof uses a matrix approach. Writing the recurrence Eq. (\ref{ec:5}) as $$\left[\begin{array}{c} H_{n+1}^{(3)}\\H_{n}^{(3)} \\ H_{n-1}^{(3)}\end{array}\right]=\left[\begin{array}{ccc} r&s&t\\1&0&0 \\ 0&1&0\end{array}\right]\left[\begin{array}{c} H_{n}^{(3)}\\H_{n-1}^{(3)} \\ H_{n-2}^{(3)}\end{array}\right],\ n\geq 2$$ in matrix form leads iteratively to the matrix power equation
\begin{equation}\label{n7}
\left[\begin{array}{c} H_{n+1}^{(3)}\\H_{n}^{(3)} \\ H_{n-1}^{(3)}\end{array}\right]=\left[\begin{array}{ccc} r&s&t\\1&0&0 \\ 0&1&0\end{array}\right]^{n-1}\left[\begin{array}{c} H_{2}^{(3)}\\H_{1}^{(3)} \\ H_{0}^{(3)}\end{array}\right]
\end{equation}
which holds for $n\geq 1$. Thus,
\begin{equation}\label{n8}
H_{n}^{(3)}(a,b,c;r,s,t)=\left[\begin{array}{ccc} 0&1&0\end{array}\right]\left[\begin{array}{ccc} r&s&t\\1&0&0 \\ 0&1&0\end{array}\right]^{n-1}\left[\begin{array}{c} c\\b \\ a\end{array}\right].
\end{equation}

Let us define matrices $$\textbf{F}(\lambda)=\left[\begin{array}{ccc} 1&0&0\\0&\lambda&0\\ 0&0&\lambda^{2}\end{array}\right],\ \ \textbf{C}(\lambda)=\left[\begin{array}{ccc} \lambda&3\lambda^{2}&9\lambda^{3}\\1&0&0 \\ 0&1&0\end{array}\right].$$ Then, observing the decomposition 
\begin{equation}\label{n10}
\frac{1}{\lambda}\textbf{F}(\lambda)\textbf{C}(\lambda)\textbf{F}^{-1}(\lambda)=\left[\begin{array}{ccc} 1&3&9\\1&0&0\\ 0&1&0\end{array}\right],\ \lambda \neq 0,
\end{equation}
it follows, using Eq. (\ref{ecua:6}) as a starting point, that $T_{n}^{(3)}=H_{n}^{(3)}(0,0,1;1,3,9)$. Then, we obtain
\begin{equation}\label{ga}
\begin{aligned}
T_{n}^{(3)}&=\left[\begin{array}{ccc} 0&1&0\end{array}\right]\left[\begin{array}{ccc} 1&3&9\\1&0&0 \\ 0&1&0\end{array}\right]^{n-1}\left[\begin{array}{c} 1\\0 \\ 0\end{array}\right]\\
&=\left[\begin{array}{ccc} 0&1&0\end{array}\right]\left[\frac{1}{\lambda}\textbf{F}(\lambda)\textbf{C}(\lambda)\textbf{F}^{-1}(\lambda)\right]^{n-1}\left[\begin{array}{ccc} 1&0&0\end{array}\right]^{T}\\
&=\left[\begin{array}{ccc} 0&1&0\end{array}\right]\left[\lambda^{-(n-1)}\textbf{F}(\lambda)\textbf{C}^{n-1}(\lambda)\textbf{F}^{-1}(\lambda)\right]\left[\begin{array}{ccc} 1&0&0\end{array}\right]^{T},
\end{aligned}
\end{equation}
having employed Eqs. (\ref{n8}) and (\ref{n10}). 

Further, noting that $\left[\begin{array}{ccc} 0&1&0\end{array}\right]\textbf{F}(\lambda)= \left[\begin{array}{ccc} 0&\lambda&0\end{array}\right]$ and the relation $\textbf{F}^{-1}(\lambda)\left[\begin{array}{ccc} 1&0&0\end{array}\right]^{T}=\left[\begin{array}{ccc} 1&0&0\end{array}\right]^{T}$, the last identity can be written as 
\begin{equation}\label{n12}
\begin{aligned}
T_{n}^{(3)}&=\lambda^{-(n-2)}\left[\begin{array}{ccc} 0&1&0\end{array}\right]\textbf{C}^{n-1}(\lambda)\left[\begin{array}{ccc} 1&0&0\end{array}\right]^{T}\\
&=\frac{1}{\lambda^{n-2}}H_{n}^{(3)}(0,0,1;\lambda,3\lambda^{2},9\lambda^{3})
\end{aligned}
\end{equation}
using Eqs. (\ref{n8}) and (\ref{ga}). The proof is completed.
\end{proof}

\begin{proof}[Proof 2 of Eq. (\ref{iden2})]
This proof takes a different route. Here, we define a sequence $D_{n}^{(3)}=\lambda^{n-2}T_{n}^{(3)}$, where initial conditions are $D_{0}^{(3)}=\lambda^{-2}T_{0}^{(3)}=0$, $D_{1}^{(3)}=\lambda^{-1}T_{1}^{(3)}=0$ and $D_{2}^{(3)}=T_{2}^{(3)}=1$. Then, for all $n\geq 2$, we have
\begin{align*}
D_{n+1}^{(3)}&=\lambda^{n-1}T_{n+1}^{(3)}\\
&=\lambda^{n-1}\left(T_{n}^{(3)}+3T_{n-1}^{(3)}+9T_{n-2}^{(3)}\right)\ \ (\textrm{By Eq. (\ref{ecua:6})})\\
&=\lambda\left(\lambda^{n-2}T_{n}^{(3)}\right)+3\lambda^{2}\left(\lambda^{n-3}T_{n-1}^{(3)}\right)+9\lambda^{3}\left(\lambda^{n-4}T_{n-2}^{(3)}\right)\\
&=\lambda D_{n}^{(3)}+3\lambda^{2}D_{n-1}^{(3)}+9\lambda^{3}D_{n-2}^{(3)}.
\end{align*}
This being a third-order Horadam recurrence in Eq. (\ref{ec:5}) for the sequence $D_{n}^{(3)}$ (with $r=\lambda$, $s=3\lambda^{2}$ and $t=9\lambda^{3}$), we have, for $n\geq 0$, $$\lambda^{n-2}T_{n}^{(3)}=H_{n}^{(3)}(0,0,1;\lambda,3\lambda^{2},9\lambda^{3})$$ as required.
\end{proof}

\begin{rem}
Using $\lambda=1$ in Eq. (\ref{iden1}), we obtain $$T_{n-1}^{(3)}=\frac{1}{9}H_{n}^{(3)}(1,0,0;1,3,9).$$ 
\end{rem}

\section{Conclusions}
Geometric mean sequences, however defined, have been studied in the past, but not to any great extent. This paper considers the notion of the growth rate of what we regard as the standard geometric mean sequence, and from that develops new identities which express generalized Tribonacci numbers in terms of parameterized families of third-order Horadam numbers. As an extension of this article, future work will examine properties of the scaled version of recurrence in Eq. (\ref{e0}) or the case the higher order Horadam sequences following the idea of Larcombe and Bagdasar in \cite{Lar3}.



\begin{thebibliography}{00}

\bibitem{Bu}
M. W. Bunder, \textit{Products and powers}, The Fibonacci Quarterly 13(3) (1975), 279.
\bibitem{Cer}
G. Cerda-Morales, \textit{Quadratic Approximation of Generalized Tribonacci Numbers}, Discussiones Mathematicae--General Algebra and Applications 32(2) (2018), 227--237.
\bibitem{Cer1} 
G. Cerda-Morales, \emph{On a Generalization of Tribonacci Quaternions}, Mediterranean Journal of Mathematics 14:239 (2017), 1--12.
\bibitem{Sha} 
A. G. Shannon  and A. F. Horadam, \textit{Some Properties of Third-Order Recurrence Relations}, The Fibonacci Quarterly 10(2) (1972), 135--146.
\bibitem{Lar1} 
P. J. Larcombe, \textit{Horadam sequences: a survey update and extension}, Bull. I.C.A. 80 (2017), 99--118.
\bibitem{Lar2} 
P. J. Larcombe, O. D. Bagdasar and E. J. Fennessey, \textit{Horadam sequences: a survey}, Bull. I.C.A. 67 (2013), 49--72.
\bibitem{Lar3} 
P. J. Larcombe, O. D. Bagdasar, \textit{On a result of Bunder involving Horadam sequences: a proof and generalization}, The Fibonacci Quarterly 51 (2013), 174--176.
\bibitem{Shi} 
A. J. Shiu and C. R. Yerger, \textit{Geometric and harmonic variations of the Fibonacci sequence}, Mathematical Spectrum 41 (2009), 81--86.
\bibitem{Wa} 
M. E. Waddill and L. Sacks, \textit{Another Generalized Fibonacci Sequence}, The Fibonacci Quarterly 5(3) (1967), 209--222.

\end{thebibliography}
\end{document}